\documentclass[12pt]{amsart}
\usepackage[applemac]{inputenc}
\usepackage{bm}

\usepackage[dvipsnames,svgnames,x11names]{xcolor}
\usepackage[hyphens]{url}

\usepackage[backref=page,colorlinks=true,linkcolor=Maroon,citecolor=blue,urlcolor=blue,hypertexnames=false,linktocpage]{hyperref}
\usepackage{bookmark}
\usepackage{amsmath,thmtools,mathtools}
\mathtoolsset{showonlyrefs=true}

\usepackage{blindtext}

\newcounter{mgncount}

\usepackage{tikz}

\usepackage{fancyhdr}
\usepackage{esint}
\usepackage{enumerate}
\usepackage{pictexwd,dcpic}
\usepackage{graphicx}
\usepackage{caption}
\allowdisplaybreaks

\usepackage{graphicx}

\newtheorem*{thm-}{Theorem}
\newtheorem*{qu-}{Question}
\declaretheorem[name=Theorem,numberwithin=section]{thm}

\declaretheorem[name=Lemma,sibling=thm]{lemma}

\declaretheorem[name=Definition,style=definition,sibling=thm]{defn}

\numberwithin{equation}{section}

\newcommand{\OP}{\Lambda_p^{\phi}}

\newcommand{\bbR}{\mathbb{R}}

\newcommand{\bbS}{\mathbb{S}}

\newcommand{\la}{\lambda}

\newcommand{\si}{\sigma}

\newcommand{\De}{\Delta}

\newcommand{\cA}{\mathcal{A}}

\newcommand{\cC}{\mathcal{C}}

\newcommand{\cK}{\mathcal{K}}

\newcommand{\cM}{\mathcal{M}}

\renewcommand{\(}{\left(}
\renewcommand{\)}{\right)}

\newcommand{\eq}[1]{\begin{equation}\begin{alignedat}{2} #1 \end{alignedat}\end{equation}}

\newcommand{\q}{\quad}

\makeatletter
\@namedef{subjclassname@2020}{%
	\textup{2020} Mathematics Subject Classification}
\makeatother

\begin{document}
	\title[Capillary curvature images]
	{Capillary curvature images}
	\author[Y. Hu, M. N. Ivaki]{Yingxiang Hu, Mohammad N. Ivaki}

\begin{abstract}
In this paper, we solve the even capillary $L_p$-Minkowski problem for the range $-n < p < 1$ and $\theta \in (0,\frac{\pi}{2})$. Our approach is based on an iterative scheme that builds on the solution to the capillary Minkowski problem (i.e., the case $p = 1$) and leverages the monotonicity of a class of functionals under a family of capillary curvature image operators. These operators are constructed so that their fixed points, whenever they exist, correspond precisely to solutions of the capillary $L_p$-Minkowski problem.
\end{abstract}

\maketitle

\section{Introduction}
Let $(\mathbb{R}^{n}, \delta = \langle \,,\, \rangle, D)$ be the Euclidean space, equipped with its standard inner product and flat connection. Denote by $(\mathbb{S}, \bar{g}, \bar{\nabla})$ the unit sphere with its induced round metric and Levi-Civita connection.

Denote by $\{E_i\}_{i=1}^{n}$ the standard orthonormal basis of $\mathbb{R}^n$, and define the upper halfspace $\mathbb{R}^n_{+} = \{ y \in \mathbb{R}^n : \langle y, E_n \rangle > 0 \}$.

Let $\Sigma$ be a properly embedded, smooth, compact, connected, and orientable hypersurface contained in the closure $\overline{\mathbb{R}^{n}_+}$, such that its interior lies entirely in $\mathbb{R}^n_+$ and its boundary satisfies $\partial \Sigma \subset \partial \mathbb{R}^{n}_+$. We say that $\Sigma$ is a capillary hypersurface with constant contact angle $\theta \in (0, \pi)$ if the following condition holds along $\partial \Sigma$:
\begin{equation}\label{capillary angle condition}
\langle \nu, E_n \rangle \equiv \cos \theta.
\end{equation}Here, $\nu$ denotes the outer unit normal of $\Sigma$. The enclosed region by the hypersurface $\Sigma$ and the hyperplane $\partial \mathbb{R}^{n}_+$ is denoted by $\widehat{\Sigma}$.  We say $\Sigma$ is strictly convex if $\widehat{\Sigma}$ is a convex body, and the second fundamental form of $\Sigma$ is positive-definite. Throughout this paper, we work only with smooth ($C^{\infty}$), strictly convex capillary hypersurfaces.

The capillary spherical cap of radius $r$ and intersecting with $\partial \mathbb{R}^{n}_+$ at a constant angle $\theta \in (0, \pi)$ is defined as
\eq{
\mathcal{C}_{\theta, r} := \left\{ x \in \overline{\mathbb{R}_+^{n}} \mid |x + r \cos \theta E_{n}| = r \right\}.
}
For simplicity, we put $\mathcal{C}_\theta = \mathcal{C}_{\theta, 1}$.

Let $\tilde{\nu}=\nu-\cos\theta E_{n}:\Sigma\to  \cC_{\theta}$ denote the capillary Gauss map of $\Sigma$, which is a diffeomorphism (cf. \cite[Lem. 2.2]{MWWX24}). The capillary support function of $\Sigma$, $s=s_{\Sigma}: \cC_{\theta}\to \bbR$, is defined as
\eq{
s(\zeta) =\langle \tilde{\nu}^{-1}(\zeta), \zeta + \cos\theta E_{n} \rangle,\q \zeta \in \cC_{\theta}.
}

We write $\hat{s}=\hat{s}_{\widehat{\Sigma}}: \bbS\to \bbR$ for the (standard) support function of the convex body $\widehat{\Sigma}$, which is defined as
\eq{
\hat{s}(u):=\max_{x\in \widehat{\Sigma}} \langle u,x\rangle,\q \forall u\in \bbS.
}

We say a function $\phi \in C(\cC_{\theta})$ is even if 
\[
\phi(-\zeta_1,\ldots, -\zeta_{n-1},\zeta_{n}) = \phi(\zeta_1,\ldots,\zeta_{n-1}, \zeta_{n}), \quad \forall \zeta \in \cC_{\theta}.
\]
Similarly, a capillary hypersurface is said to be even if its capillary support function is even.

Let $\sigma$ denote the spherical Lebesgue measure on $\cC_{\theta}$. The following existence result was established in \cite{MWW25}. For a streamlined derivation of the $C^2$-norm bound required in its proof, see also \cite[Lem. 4.9]{HIS25}.

\begin{thm}[Capillary Minkowski problem in halfspace]\label{Min P}
Let $\theta \in (0, \frac{\pi}{2})$. Suppose $0 < \phi \in C^2(\cC_\theta)$ satisfies  
\begin{equation}
\int_{\cC_\theta} \langle \zeta, E_{i} \rangle \phi(\zeta) \, d\sigma(\zeta) = 0, \quad i = 1,2,\ldots,n-1,  
\end{equation}  
where $\{E_{i}\}_{i=1,\ldots,n-1}$ is the horizontal basis of $\partial \mathbb{R}^{n}_+$.  Then there exists a $C^{3,\alpha}$ strictly convex capillary hypersurface $\Sigma \subset \overline{\mathbb{R}^{n}_{+}}$ such that its Gauss-Kronecker curvature $\mathcal{K}$ satisfies  
\begin{equation} \label{capillary-Minkowski-problem}
\frac{1}{\mathcal{K}(\tilde{\nu}^{-1}(\zeta))} = \phi(\zeta), \quad \forall \zeta \in \cC_\theta.
\end{equation}  
Moreover, $\Sigma$ is unique up to a horizontal translation in $\overline{\mathbb{R}^{n}_{+}}$.  
\end{thm}

Let $\bbS_{\theta} := \{y \in \bbS : y_n \geq \cos\theta\}$ and define $T$ by
\eq{
T: \bbS_{\theta} \to \cC_{\theta}, \quad T(u) = u - \cos\theta E_{n}.
}
The surface area measure of $\widehat{\Sigma}$ in \autoref{Min P} can be determined by
\eq{
S(\widehat{\Sigma}, \omega) = \mu := 
\begin{cases}
	\int_{\omega} \phi \circ T \, d\sigma, & \omega \subseteq \bbS_{\theta}, \\
	\int_{\bbS_{\theta}} u_n \phi \circ T(u) \, d\sigma(u), & \omega = \{-E_{n}\}, \\
	0, & \omega \subseteq \bbS \setminus \left(\{-E_{n}\} \cup \bbS_{\theta}\right).
\end{cases}
}

Note that by applying the standard Minkowski existence theorem, we can find a convex body $K$ whose surface area measure is $\mu$. However, a priori, it is not evident to us that a translation of $K$ satisfies the capillary angle condition \eqref{capillary angle condition}, and hence yields a capillary convex body. Nevertheless, when $\theta \in (0, \frac{\pi}{2})$, in view of \autoref{Min P}, such a translation of $K$ does exist. It would be interesting to establish this fact directly from the definition of $\mu$, in particular, for all $\theta \in (0, \pi)$ and under a milder regularity assumption $\phi \in C^{\alpha}(\cC_{\theta})$.

The Minkowski problem, rooted in the seminal work of Hermann Minkowski \cite{Min97,Min03}, is a foundational question in convex geometry. It seeks to determine whether a given finite Borel measure on the sphere arises as the surface area measure of a convex body, and if so, whether such a body is unique. Foundational results by Aleksandrov \cite{Ale56}, Nirenberg \cite{Nir57}, Pogorelov \cite{Pog52,Pog71}, and Cheng--Yau \cite{CY76} established the existence, uniqueness, and regularity of solutions to this classical problem.

A significant generalization of the classical Minkowski problem is the $L_p$-Minkowski problem for $p > 1$, introduced by Lutwak \cite{Lut93} within the framework of the $L_p$-Brunn--Minkowski--Firey theory. The even case (i.e. origin-symmetric case) was subsequently solved by Lutwak and Oliker \cite{LO95}. When $p < 1$, and particularly in the cases $p = 0$ and $p = -n$, new difficulties arise due to the breakdown of standard techniques such as the continuity method and lack of compactness. Substantial progress in this range has been made by Chou and Wang \cite{CW06}, B{\"o}r{\"o}czky, Lutwak, Yang and Zhang \cite{BLYZ13}, and Bianchi et al. \cite{BBCY19}. For further developments, we refer the reader to \cite{ACW01, Sta02, TW08, HLYZ16, Li19, HXY21, GLW22, GLW24, LXYZ24} and the recent surveys \cite{Bor23, HYZ25}.

Geometric flows have also emerged as a powerful tool for investigating Minkowski-type problems. The logarithmic Gauss curvature flow \cite{CW00} and its $L_p$ analogues \cite{BG23, BIS21, CL21, LWW20, BIS19} have played a role in the analysis of both classical and dual formulations. More recently, Guang et al. \cite{GLW22} extended these techniques to the supercritical range $p < -n$.

The scope of Minkowski-type problems continues to grow, including recent extensions into the capillary setting in the halfspace \cite{MWW25} for $p = 1$ and \cite{MWW25b} for $p > 1$, both for $\theta \in \left(0, \frac{\pi}{2}\right)$. Further developments within the capillary framework can be found in \cite{KLS25, MWWX24, HWYZ24, WWX24, SW24, MW23, WW20, WX19}. In the present paper, we obtain the following existence result for the even capillary $L_p$-Minkowski problem for $-n < p < 1$:

\begin{thm}\label{Min LP2}
Let $\theta \in \left(0, \frac{\pi}{2}\right)$, $-n < p < 1$, and let $0 < \phi \in C^{\infty}(\cC_\theta)$ satisfy  
\[
\phi(-\zeta_1,\ldots, -\zeta_{n-1},\zeta_{n}) = \phi(\zeta_1,\ldots,\zeta_{n-1}, \zeta_{n}), \quad \forall \zeta \in \cC_{\theta}.
\]
Then there exists an even, smooth, strictly convex capillary hypersurface $\Sigma$ whose capillary support function $s$ and Gauss curvature $\mathcal{K}$ satisfy
\begin{equation}
\frac{s^{1-p}}{\mathcal{K} \circ \tilde{\nu}^{-1}} = \phi.
\end{equation}
\end{thm}

It is important to note that although the standard continuity method is effective in the range $p \geq 1$, it generally fails for $p \in (-n,1)$ due to the lack of openness. Likewise, degree-theoretic approaches -- while successful in the classical setting -- are currently not viable in the capillary context due to the absence of suitable uniqueness results under Robin boundary conditions.

A variational method -- similar to those employed in the standard case of the $L_p $-Minkowski problem, where one seeks to maximize a functional over a suitable class of convex bodies -- does not readily apply here. Although one might attempt to formulate such a functional within the class of strictly convex capillary hypersurfaces, it is not clear to us a priori that an extremal hypersurface, whenever it exists, would be smooth and satisfies the constant contact angle condition: the lack of regularity prevents direct application of the Euler-Lagrange multiplier method via Aleksandrov's variational lemma (cf. \cite[Thm. 7.5.3]{Sch14}) and thus obstructs such a standard variational treatment.

While one possible route to establishing \autoref{Min LP2} would be through the construction of a suitable curvature flow -- see, for instance, \cite{BIS19, BG23} -- in this work, we pursue a different strategy which can be considered as more of a discrete flow. Similar to the works \cite{Iva16, Iva20}, we construct a class of new curvature image operators using the existence result of \autoref{Min P}, whose fixed points -- when they exist -- correspond to solutions of the capillary $L_p$-Minkowski problem.

Iteratively applying such curvature image operator to an initial even, strictly convex capillary hypersurface -- such as $\cC_\theta$ -- generates a sequence of even, strictly convex capillary hypersurfaces. The structural properties of our curvature image operator ensure that this sequence remains within a compact family, thereby allowing us to extract a convergent subsequence. By invoking the equality characterization in the Minkowski inequality (in place of Aleksandrov's variational lemma), we conclude that the limit hypersurface solves the even capillary $L_p$-Minkowski problem. The uniform $C^{m}$-norm bound (for all $m\geq 2$) along the iteration follows from the recent work \cite[Lem. 4.9]{HIS25} and by establishing a recursive inequality.

\section{Capillary curvature image operators}

\begin{defn}\label{def curv op} Let $\theta \in \left(0, \frac{\pi}{2}\right)$. Suppose $\Sigma$ is a strictly convex capillary hypersurface. We define the capillary curvature function of $\Sigma$ as
\eq{
f_{\Sigma} &:\cC_{\theta} \to (0,\infty), \quad 
f_{\Sigma}(\zeta) = \frac{1}{\cK \circ \tilde{\nu}^{-1}(\zeta)}.
}

Let $0 < \phi \in C^{\infty}(\cC_{\theta})$ be an even function and $\Sigma$ be an even strictly convex capillary hypersurface. We define the capillary curvature image $\OP\Sigma$ of $\Sigma$ as the unique, even strictly convex capillary hypersurface whose curvature function is
\eq{\label{curvature image}
f_{\OP\Sigma} = \frac{V(\widehat{\Sigma})}{\frac{1}{n} \int_{\cC_{\theta}} \phi s^p_{\Sigma} \, d\sigma} \phi s_{\Sigma}^{p-1}.
}
\end{defn}

Note that, in view of \autoref{Min P}, the operator $\OP$ is well-defined. The definition of our capillary curvature image operator $\OP\Sigma$ is motivated by some of the properties of Petty's curvature image operator, discovered by Lutwak in the classical setting; see \cite{Lut86} and \cite[Sec. 10.5]{Sch14}.

Note that $V(\widehat{\Sigma}) = V(\widehat{\Sigma}, \widehat{\OP\Sigma}[n-1])$ (see \cite{MWWX24} for the definition of the capillary mixed volume). Therefore, by \cite[Thm. 1.1]{MWWX24}, we obtain
\eq{\label{key}
V(\widehat{\Sigma}) \geq V(\widehat{\OP \Sigma}),
}
and equality holds if and only if $\Sigma = \OP\Sigma$. Here, the notation $\widehat{\OP\Sigma}[n-1]$ indicates that $\widehat{\OP\Sigma}$ appears $(n-1)$ times in the mixed volume.

\begin{defn}\label{def}
Let $\theta\in (0, \frac{\pi}{2})$ and $0<\phi\in C^{\infty}(\cC_{\theta})$ be an even function. Suppose $\Sigma$ is an even strictly convex capillary hypersurface. We define
\eq{
\mathcal{A}_p^{\phi}(\Sigma) &=
\begin{cases}
V(\widehat{\Sigma})\left(\int_{\cC_{\theta}}\phi s_{\Sigma}^p\, d\sigma\right)^{-\frac{n}{p}}, & p\neq 0,\\
V(\widehat{\Sigma})\exp\left(\frac{\int_{\cC_{\theta}} -\phi \log s_{\Sigma} \, d\sigma}{\frac{1}{n} \int_{\cC_{\theta}} \phi \, d\sigma}\right), & p=0,
\end{cases}\\
\mathcal{B}_p^{\phi}(\Sigma)&=
\begin{cases}
V(\widehat{\Sigma})^{1-n}\left(\int_{\cC_{\theta}} \phi^{-\frac{1}{p-1}} f_{\Sigma}^{\frac{p}{p-1}}\, d\sigma\right)^{\frac{n(p-1)}{p}}, & p\neq 0,1,\\
V(\widehat{\Sigma})^{1-n} \exp\left(\int_{\cC_{\theta}} \log(\frac{f_{\Sigma}}{\phi} )\, d\tilde{\sigma}\right)^n \left(\int_{\cC_{\theta}} \phi\, d\sigma\right)^n, & p=0,
\end{cases}\\
\Omega_p^{\phi}(\Sigma) &=
\begin{cases}
\int_{\cC_{\theta}} \phi^{-\frac{1}{p-1}} f_{\Sigma}^{\frac{p}{p-1}}\, d\sigma, & p\neq 0,1, \\
\exp\left(\frac{\int_{\cC_{\theta}} \phi \log f_{\Sigma}\, d\sigma}{\frac{1}{n} \int_{\cC_{\theta}} \phi\, d\sigma}\right), & p=0.
\end{cases}
}
Here $d\tilde{\sigma} := \dfrac{\phi}{\int_{\cC_{\theta}} \phi\, d\sigma}\, d\sigma$.
\end{defn}

\begin{lemma}Let $\theta\in (0,\frac{\pi}{2})$ and $0<\phi\in C^{\infty}(\cC_{\theta})$ be an even function. Suppose $\Sigma$ is an even strictly convex capillary hypersurface. Then
\eq{\label{key identity}
\mathcal{B}_p^{\phi}(\Lambda_p^{\phi}\Sigma)=n^{n}\left(\frac{V(\widehat{\Sigma})}{V(\widehat{\OP\Sigma})}\right)^{n-1}\mathcal{A}_p^{\phi}(\Sigma),
}
and
\eq{\label{key inequalities}
\begin{cases}
\mathcal{B}_p^{\phi}(\Sigma)\leq n^n \cA_p^{\phi}(\Sigma), & p<1,\\
\mathcal{B}_p^{\phi}(\Sigma)\geq n^n \cA_p^{\phi}(\Sigma), & p>1.
\end{cases}
}
\end{lemma}
\begin{proof}
Verification of the first identity is straightforward. For the inequalities, note that 
\[
V(\widehat{\Sigma}) = \frac{1}{n} \int_{\cC_{\theta}} s_{\Sigma} f_{\Sigma}\, d\sigma.
\]
The inequalities then follow from the H\"{o}lder inequality when $p \neq 0$, and from the Jensen inequality when $p = 0$.
\end{proof}

\begin{lemma}\label{C0 estimate-step 1}
	Let $\theta\in (0, \frac{\pi}{2})$, $-n\leq p<1$ and $0<\phi\in C^{\infty}(\cC_{\theta})$ be an even function. Suppose $\Sigma$ is an even strictly convex capillary hypersurface. Then there exists a constant $a_p^{\phi}$ such that
\eq{
\mathcal{A}_p^{\phi}(\Sigma)\leq a_p^{\phi}, \q \mathcal{B}_p^{\phi}(\Sigma)\leq n^n a_p^{\phi}.
}
\end{lemma}
\begin{proof} In view of \eqref{key inequalities}, we need only to prove $\mathcal{A}_p^{\phi}(\Sigma)\leq a_p^{\phi}$.
Let $K$ denote the convex body enclosed by $\Sigma$ and its reflection across $x_{n}=0$. Note that $K$ is origin-symmetric. Let $\hat{s}_K$ denote the support function of $K$. Define 
 $\psi: \bbS\to (0,\infty)$ by
\eq{
\psi(u)=
\begin{cases}
	\phi\circ T(u),& u_{n}\geq \cos\theta,\\
	\phi\circ T(u_1,\ldots,u_{n-1}, -u_{n}),& u_{n}\leq -\cos\theta,\\
	\phi\circ T(\sin \theta\, \tilde{u}, \cos \theta),& -\cos\theta \leq u_{n} \leq \cos\theta,
\end{cases}
}
where $\tilde{u}:=\frac{u-\langle u,E_{n}\rangle E_{n}}{|u-\langle u,E_{n}\rangle E_{n}|}$.

Note that $\hat{s}_K(u)=s_{\Sigma}\circ T(u)$ for $u\in \bbS_{\theta}$. Hence, for $-n\leq p<0$, we have
\eq{
V(K)\left(\int_{\bbS}\psi \hat{s}_{K}^p\, d\sigma\right)^{-\frac{n}{p}}\geq \cA_p^{\phi}(\Sigma),
}
while for $0<p<1$, by the H\"{o}lder inequality, there holds
\eq{
\cA_p^{\phi}(\Sigma)\leq \left(\int_{\cC_{\theta}}\phi \, d\si\right)^{-\frac{2n}{p}}\cA_{-p}^{\phi}(\Sigma)\leq \left(\int_{\cC_{\theta}}\phi\, d\si\right)^{-\frac{2n}{p}} V(K)\left(\int_{\bbS}\psi \hat{s}_{K}^{-p}\, d\sigma\right)^{\frac{n}{p}}.
}
Therefore, the claim follows from the Blaschke-Santal\'{o} inequality:
\eq{
V(K)\left(\frac{1}{n}\int_{\bbS} \frac{1}{\hat{s}_K^n}\, d\si\right)\leq V(B)^2,
}
where $B$ denotes the unit ball.

The claim for the $p=0$ case follows from the Jensen inequality and the Blaschke-Santal\'{o} inequality:
\eq{
\cA_{0}^{\phi}(\Sigma)\leq \frac{\cA_{-1}^{\phi}(\Sigma)}{\left(\int_{\cC_{\theta}}\phi \, d\si\right)^n}\leq  \frac{V(K)}{\left(\int_{\cC_{\theta}}\phi \, d\si\right)^n}\left(\int_{\bbS} \frac{\psi}{\hat{s}_{K}}\, d\sigma\right)^{n}.
}\end{proof}

\begin{lemma}\label{lem1}Let $\theta\in (0, \frac{\pi}{2})$, $-n\leq p<1$ and $0<\phi\in C^{\infty}(\cC_{\theta})$ be an even function. Suppose $\Sigma$ is an even strictly convex capillary hypersurface. The following inequalities hold:
\begin{enumerate}
  \item $\mathcal{A}_p^{\phi}(\Sigma)\leq \left(\frac{V(\widehat{\OP\Sigma})}{V(\widehat{\Sigma})}\right)^{n-1}\mathcal{A}_p^{\phi}(\OP \Sigma)\leq \mathcal{A}_p^{\phi}(\OP \Sigma).$
  \item If $p\neq 0$, then 
   \begin{align}\label{ineq-Omega-neq-0} 
   \Omega_p^{\phi}(\Sigma)^{\frac{p-1}{p}}\leq \Omega_p^{\phi}(\OP \Sigma)^{\frac{p-1}{p}}.
   \end{align}
  If $p=0$, then
  \begin{align}\label{ineq-Omega-p=0}
  \Omega_0^{\phi}(\Sigma)\leq \Omega_0^{\phi}(\Lambda_0^{\phi} \Sigma).
  \end{align}
  \item If $p\neq 0$, then 
   \eq{\label{ineq-Omega-volume-p-neq-0} 
   c_{p}^{\phi}\Omega_p^{\phi}(\Sigma)^{\frac{n(p-1)}{p(n-1)}}\leq V(\widehat{(\OP)^i \Sigma})\leq V(\widehat{\Sigma}), \quad \forall i.} 
 If $p=0,$ then
  \eq{\label{ineq-Omega-volume-p=0} c_0^{\phi}\Omega_0^{\phi}(\Sigma)^{\frac{1}{n-1}}\leq V(\widehat{(\Lambda_0^{\phi})^i \Sigma})\leq V(\widehat{\Sigma}),\quad \forall i.}
\end{enumerate}
\end{lemma}
\begin{proof}
The left-hand side inequality in (1) can be rewritten as
\[
\mathcal{B}_p^{\phi}(\Lambda_p^{\phi}\Sigma) \leq n^n \mathcal{A}_p^{\phi}(\OP \Sigma),
\]
which follows directly from \eqref{key inequalities}. The right-hand side inequality in (1) is a consequence of \eqref{key}.

For the inequality \eqref{ineq-Omega-neq-0} in the case $p \neq 0$, recall that
\eq{
\Omega_p^{\phi}(\Sigma) = \int_{\cC_\theta} \phi^{-\frac{1}{p-1}} f_{\Sigma}^{\frac{p}{p-1}} \, d\sigma,
}
and
\eq{\Omega_p^{\phi}(\Lambda_p^\phi \Sigma)&=\int_{\cC_\theta}\phi^{-\frac{1}{p-1}}f_{\Lambda_p^\phi \Sigma}^\frac{p}{p-1} \,d\sigma\\
&=n^\frac{p}{p-1} \(\int_{\cC_\theta} \phi s_{\Sigma}^p \,d\sigma\)^{-\frac{1}{p-1}} V(\widehat{\Sigma})^{\frac{p}{p-1}}\\
&=\(\int_{\cC_\theta} \phi s_{\Sigma}^p \,d\sigma\)^{-\frac{1}{p-1}} \(\int_{\cC_\theta} s_\Sigma f_\Sigma\,d\sigma\)^{\frac{p}{p-1}}.
}
Then \eqref{ineq-Omega-neq-0} follows from the H\"older inequality. For $p=0$, we have
\eq{
\Omega_0^{\phi}(\Sigma)=\exp\(\frac{\int_{\cC_\theta}\phi \log f_\Sigma \,d\sigma}{\frac{1}{n}\int_{\cC_\theta}\phi \,d\sigma}\),
}
and
\eq{
\Omega_0^{\phi}(\Lambda_0^\phi\Sigma)=\exp\left(\frac{\int_{\cC_\theta}\phi \log (\frac{V(\widehat{\Sigma})}{\frac{1}{n}\int_{\cC_\theta}\phi \,d\sigma}\frac{\phi}{s_\Sigma}) \,d\sigma}{\frac{1}{n}\int_{\cC_\theta}\phi \,d\sigma}\right).
}
Then the inequality \eqref{ineq-Omega-p=0} is equivalent to 
\eq{
\int_{\cC_\theta}\phi \log f_\Sigma \,d\si &\leq \int_{\cC_\theta}\phi \log (\frac{\int_{\cC_\theta} s_{\Sigma}f_\Sigma\,d\sigma}{\int_{\cC_\theta}\phi \,d\sigma}\frac{\phi}{s_\Sigma}) \,d\sigma\\
&=\left(\int_{\cC_\theta}\phi \,d\sigma\right) \log \frac{\int_{\cC_\theta} s_\Sigma f_\Sigma \,d\sigma}{\int_{\cC_\theta}\phi \,d\sigma}-\int_{\cC_\theta}  \phi\log (\frac{s_{\Sigma}}{\phi})\,d\sigma.
}
which follows from the Jensen inequality:
\eq{
\frac{\int_{\cC_\theta} \log (\frac{s_\Sigma f_{\Sigma}}{\phi})\phi\,d\sigma}{\int_{\cC_\theta}\phi \,d\sigma}\leq \log (  \frac{\int_{\cC_\theta}\frac{ s_\Sigma f_\Sigma}{\phi} \phi\,d\sigma}{\int_{\cC_\theta}\phi \,d\sigma}).
}

To show the left-hand side of the inequality \eqref{ineq-Omega-volume-p-neq-0} in (3) for $p\neq 0$, by using (2) iteratively $(i+1)$-times, we have
\eq{
\Omega_p^{\phi}(\Sigma)^\frac{n(p-1)}{p(n-1)}&\leq \Omega_p^\phi((\Lambda_p^\phi)^{i+1}\Sigma)^\frac{n(p-1)}{p(n-1)}\\
&=\(\int_{\cC_\theta}\phi s_{(\Lambda_p^\phi)^{i}\Sigma}^p \,d\sigma \)^{-\frac{n}{p(n-1)}}\( \int_{\cC_\theta} s_{(\Lambda_p^\phi)^{i}\Sigma}f_{(\Lambda_p^\phi)^{i}\Sigma}\,d\sigma  \)^\frac{n}{n-1}\\
&=n^\frac{n}{n-1}\mathcal{A}_p^\phi( (\Lambda_p^\phi)^{i}\Sigma)^\frac{1}{n-1} V(\widehat{(\Lambda_p^\phi)^{i}\Sigma})\\
&\leq n^\frac{n}{n-1}(a_p^\phi)^\frac{1}{n-1} V(\widehat{(\Lambda_p^\phi)^{i}\Sigma})=:(c_p^{\phi})^{-1}V(\widehat{(\Lambda_p^\phi)^{i}\Sigma}),
  }
where we used \autoref{C0 estimate-step 1} in the last inequality. The right-hand side of \eqref{ineq-Omega-volume-p-neq-0} in (3) then follows from
\eq{
V(\widehat{(\Lambda_p^\phi)^i\Sigma})\leq V(\widehat{(\Lambda_p^\phi)^{i-1}\Sigma})\leq \cdots \leq V(\widehat{\Sigma}),
}
due to \eqref{key}. 

Similarly, for the case $p=0$, the left-hand side of \eqref{ineq-Omega-volume-p=0} can be deduced from \eqref{ineq-Omega-p=0} and \autoref{C0 estimate-step 1}:
\eq{
\Omega_0^\phi(\Sigma) &\leq \Omega_0^\phi((\OP)^{i+1}\Sigma)\\
&=\exp\left(\frac{\int_{\cC_\theta}\phi \log(\frac{V(\widehat{(\OP)^i\Sigma})}{\frac{1}{n}\int_{\cC_\theta}\phi \,d\sigma}\frac{\phi}{s_{(\OP)^i\Sigma}})\,d\sigma}{\frac{1}{n}\int_{\cC_\theta}\phi \,d\sigma}\right)\\
&=V(\widehat{(\OP)^i\Sigma})^{n-1} \(\frac{1}{n}\int_{\cC_\theta}\phi \,d\sigma\)^{-n} \exp\(\frac{\int_{\cC_\theta}\phi\log\phi \,d\sigma}{\frac{1}{n}\int_{\cC_\theta}\phi\,d\sigma}  \) \mathcal{A}_0^\phi((\Lambda_0^\phi)^i\Sigma)\\
&\leq V(\widehat{(\OP)^i\Sigma})^{n-1} \(\frac{1}{n}\int_{\cC_\theta}\phi \,d\sigma\)^{-n} \exp\(\frac{\int_{\cC_\theta}\phi\log\phi\, d\sigma}{\frac{1}{n}\int_{\cC_\theta}\phi\,d\sigma}  \) a_0^\phi\\
&=: (c_0^\phi)^{-(n-1)} V(\widehat{(\OP)^i\Sigma})^{n-1}.
}
The right-hand side of \eqref{ineq-Omega-volume-p=0} follows again from \eqref{key}.
\end{proof}

\begin{lemma}\label{C0 estimate}
	Let $\theta\in (0, \frac{\pi}{2})$, $-n<p<1$ and $0<\phi\in C^{\infty}(\cC_{\theta})$ be an even function. Suppose $\Sigma$ is an even strictly convex capillary hypersurface with $c_1\leq V(\widehat{\Sigma})\leq c_2$ and $ \cA_p^{\phi}(\Sigma)\geq c_3$. Then there exists $C=C(c_i)$ such that
	\eq{
C^{-1}\leq s_{\Sigma}\leq C.
	}
\end{lemma}
\begin{proof}
Define $q(p)=p$ for $p<0$, $q(p)=-p$ for $0<p<1$ and $q(p)=-1$ for $p=0$. Then from the proof  of \autoref{C0 estimate-step 1}, it follows that
\eq{
	V(K)\left(\int_{\bbS}\psi \hat{s}_{K}^{q(p)}\, d\sigma\right)^{-\frac{n}{q(p)}}\geq c_4,
}
where $c_4$ depends on $c_3$ and $\phi$.
Now by \cite[Lem. 6.1]{Iva16}, the claim follows.
\end{proof}

\section{An iterative scheme}

\begin{lemma}\label{uniform C1 estimate}
Let $\theta\in (0,\frac{\pi}{2})$, $-n<p<1$, and $0<\phi\in C^{\infty}(\cC_{\theta})$ be an even function. 
Suppose $\Sigma$ is an even strictly convex capillary hypersurface. Then for some constant $C$ depending  on $\theta, \Sigma$ and $\phi$ we have 
\eq{\label{C1 estimate}
\frac{1}{C}\leq s_{(\OP)^i \Sigma}\leq C,\q |\hat{s}_{\widehat{(\OP)^i \Sigma}}|\leq C.
}
Moreover, for each $m\geq 1$, there exists $C_m$ depending on $\theta, \Sigma$ and $\phi$ such that
\eq{
\|s_{(\OP)^i \Sigma}\|_{C^{m}(\cC_{\theta})}\leq C_m,\q \forall i.
}
\end{lemma}

\begin{proof}
By \autoref{C0 estimate-step 1} and \autoref{lem1}, for some constants $c_i$, we have
\eq{
c_1\leq V(\widehat{(\OP)^i \Sigma})\leq c_2,\q \cA_p^{\phi}((\OP)^i \Sigma)\geq c_3 ,\q  \forall i.
}
Hence, by \autoref{C0 estimate}, we have
\eq{
\frac{1}{C}\leq s_{(\OP)^i \Sigma}\leq C,\q \forall i.
}
From \cite[Lem. 4.2]{HIS25} and \cite[Lem. 4.8]{HIS25}, it follows that
\eq{
\|s_{(\OP)^i \Sigma}\|_{C^1(\cC_{\theta})}\leq C, \q |\hat{s}_{\widehat{(\OP)^i \Sigma}}|\leq C,
}
for a (new) constant $C$.

Next, we need to obtain the uniform $C^2$-norm bound. We consider the case $n \geq 3$. Let us define $\tau[\xi] := \bar{\nabla}^2 \xi + \xi \bar{g}$, and let $\sigma_k(\tau^{\sharp}[\xi])$ denote the $k$-th elementary symmetric function of the eigenvalues of $\tau[\xi]$ with respect to the metric $\bar{g}$. Note that at each iteration step $i\geq 0$, we are solving 
\eq{
\left\{
\begin{aligned}
\si_{n-1}(\tau^{\sharp}[\xi]) &= \frac{V(\widehat{(\OP)^i\Sigma})}{\frac{1}{n}\int_{\cC_\theta} \phi s_{(\OP)^i\Sigma}^p \,d\sigma} \phi s_{(\OP)^i\Sigma}^{p-1}, &\quad \text{in } \cC_\theta, \\
\bar\nabla_\mu \xi &= \cot\theta \, \xi, &\quad \text{on } \partial \cC_\theta
\end{aligned}
\right.
}
whose solution by \autoref{Min P} exists and is denoted by $s_{(\OP)^{i+1}\Sigma}$.

For simplicity, we set $s_i = s_{(\OP)^i\Sigma}$ with $s_0=s_{\Sigma}$ and 
\eq{
\gamma_i = \frac{V(\widehat{(\OP)^i\Sigma})}{\frac{1}{n}\int_{\cC_\theta} \phi s_i^p \,d\sigma}, \qquad f = \left(\gamma_i \phi s_i^{p-1}\right)^{\frac{1}{n-1}}.
}

We follow the approach in the proof of \cite[Lem. 4.9]{HIS25}, with particular attention to ensuring that the $C^2$-estimate remains uniformly controlled with respect to the iteration index $i$.

Let $F = \sigma_{n-1}^{\frac{1}{n-1}}$. If the maximum of $\si_1(\tau^{\sharp}[s_{i+1}])$ is attained in the interior of $\cC_\theta$, then
\eq{
0 \geq F^{ij}\bar\nabla^2_{i,j} \si_1 &= -\bar g^{kl} F^{ij,pq} \bar\nabla_k \tau_{ij} \bar\nabla_l \tau_{pq} - (n-1)f + \bar\De f + F^{ij} \bar g_{ij} \si_1 \\
&\geq -(n-1)f + \bar\De f + F^{ij} \bar g_{ij} \si_1,
}
where we used $\bar g^{kl}F^{ij,pq} \bar\nabla_k \tau_{ij} \bar\nabla_l \tau_{pq} \leq 0$ and the one-homogeneity and concavity of $F$. Moreover, we have
\eq{
F^{ij} \bar g_{ij} = \frac{1}{n-1}\si_{n-1}^{\frac{2-n}{n-1}} \si_{n-2} &\geq c_n \si_{n-1}^{-\frac{1}{(n-1)(n-2)}} \si_1^{\frac{1}{n-2}}= c_n f^{-\frac{1}{n-2}} \si_1^{\frac{1}{n-2}}.
}
Hence, we obtain
\eq{
\si_1(\tau^{\sharp}[s_{i+1}])^{\frac{n-1}{n-2}} \leq c_n^{-1} \gamma_i^{\frac{1}{n-1}} f^{\frac{1}{n-2}}\left( (n-1) \phi^{\frac{1}{n-1}} s_i^{\frac{p-1}{n-1}} - \bar\De(\phi^{\frac{1}{n-1}} s_i^{\frac{p-1}{n-1}}) \right).
}
Note that
\eq{
- \bar\De(\phi^{\frac{1}{n-1}} s_i^{\frac{p-1}{n-1}}) &= - s_i^{\frac{p-1}{n-1}} \bar\De \phi^{\frac{1}{n-1}} - 2 \langle \bar\nabla \phi^{\frac{1}{n-1}}, \bar\nabla s_i^{\frac{p-1}{n-1}} \rangle - \phi^{\frac{1}{n-1}} \bar\De s_i^{\frac{p-1}{n-1}} \\
&\leq c_0 - \frac{p-1}{n-1} \phi^{\frac{1}{n-1}} s_i^{\frac{p-n}{n-1}} \bar\De s_i \\
&\leq c_1 + \frac{1-p}{n-1} \phi^{\frac{1}{n-1}} s_i^{\frac{p-n}{n-1}} \si_1(\tau^{\sharp}[s_i]),
}
where we used the uniform $C^1$-bound on $s_i$, the uniform lower bound on $\min_{\cC_\theta} s_i$ from \autoref{uniform C1 estimate}, and the $C^2$-bound on $\phi$. Thus,
\eq{\label{C2-bound iteration}
\si_1(\tau^{\sharp}[s_{i+1}])^{\frac{n-1}{n-2}} \leq c_1 + c_2 \si_1(\tau^{\sharp}[s_i]),
}
where $c_1$, $c_2$ are positive constants depending only on $\theta,\Sigma,\phi$.

Now suppose the maximum of $\si_1(\tau^{\sharp}[s_{i+1}])$ is attained at a boundary point $q \in \partial \cC_\theta$. Then at $q$,
\eq{
0 \leq F^{\mu\mu} \bar\nabla_\mu \si_1 &= \cot\theta \sum_i (F^{\mu\mu} - F^{ii})(\la_\mu - \la_i) + \bar\nabla_\mu f \\
&= \cot\theta \left(f - F^{\mu\mu} \si_1 + \sum_i F^{\mu\mu} \la_\mu - \sum_i F^{ii} \la_\mu \right) + \bar\nabla_\mu f \\
&\leq \cot\theta \left(n f - F^{\mu\mu} \si_1 - \sum_i F^{ii} \la_\mu \right) + \bar\nabla_\mu f.
}
Thus, using the uniform $C^0$-bound on $s_i$ and that $\bar\nabla_\mu s_i = \cot\theta s_i$, we obtain
\eq{
\si_1(q) &\leq \frac{1}{F^{\mu\mu}(q)}\left(\frac{\max |\bar\nabla_\mu f|}{\cot\theta } + n f(q)\right) \\
&=\frac{1}{F^{\mu\mu}(q)} \left(\frac{\gamma_i^{\frac{1}{n-1}} \max |\bar\nabla_\mu(\phi^{\frac{1}{n-1}} s_i^{\frac{p-1}{n-1}})|}{ \cot\theta } + n \gamma_i^{\frac{1}{n-1}} \phi^{\frac{1}{n-1}}(q) s_i(q)^{\frac{p-1}{n-1}}\right)\\
&\leq \frac{C_0}{F^{\mu\mu}(q)} \label{key-ineq-sigma_1}.
}
To show $F^{\mu\mu}$ cannot be very small, note that
\eq{
c := \min_{\cC_\theta} (\gamma_i \phi s_i^{p-1}) \leq \si_{n-1}(\la) = \la_\mu \si_{n-2}(\la | \la_\mu) \leq c' \la_\mu F^{\mu\mu},
}
where $c'$ depends only on $\theta, \phi$ and $\Sigma$. Thus,
\eq{
0 \leq F^{\mu\mu} \bar\nabla_\mu \si_1 &\leq \cot\theta \, (n f - \sum_i F^{ii} \la_\mu) + \bar\nabla_\mu f \\
&\leq n f \cot\theta - \frac{c \cot\theta}{c' F^{\mu\mu}} \sum_i F^{ii} + \bar\nabla_\mu f.
}
Using $\sum_i F^{ii}  \geq c_n$ from the concavity of $F$, and $\bar\nabla_\mu s_i = \cot\theta s_i$ we conclude that $F^{\mu\mu}$ is uniformly bounded below. In view of \eqref{key-ineq-sigma_1}, we obtain
\eq{
\sigma_1(\tau^{\sharp}[s_{i+1}]) \leq c_3,
}
where $c_3$ depends only on $\theta, \phi$ and $\Sigma$.

Let $a_i = \max_{\cC_\theta} \si_1(\tau^{\sharp}[s_i])$ for $i\geq 0$. In summary, we have found constants $c_1, c_2, c_3$, depending only on $\theta, \phi$ and $\Sigma$, such that for all $i \geq 0$:
\eq{
\text{either} \quad a_{i+1}^{\frac{n-1}{n-2}} \leq c_1 + c_2 a_i, \quad \text{or} \quad a_{i+1} \leq c_3.
}
Let $a := \max\{1, a_0, 2c_1, 2c_2, c_3\}$. We claim that $a_i \leq a^{n-2}$. For $i = 0$, the claim holds. Suppose $i \geq 1$ is the smallest index such that $a_i > a^{n-2}$. Then,
\eq{
a^{n-1} < a_i^{\frac{n-1}{n-2}} \leq \frac{a}{2}(1 + a_{i-1}) \quad \Rightarrow \quad a^{n-2} < \frac{a_{i-1} + 1}{2}.
}
This implies that $a_{i-1} \geq 1$, and
\eq{
a^{n-2} < \frac{a_{i-1} + 1}{2} \leq a_{i-1},
}
a contradiction.

Having established the $C^2$-norm bound, the $C^{2,\alpha}$ and uniform $C^m$-norm bounds for all $m \geq 3$ follow from \cite{LT86} and the Schauder estimates.
\end{proof}

We can now apply the capillary curvature image operator $\OP$ iteratively to construct a solution to the even $L_p$-Minkowski problem. We may take $\Sigma = \cC_\theta$ (or any other even, smooth, strictly convex capillary hypersurface) as a starting point.

Due to the monotonicity (i.e., non-decreasing behavior) of $\mathcal{A}_p^{\phi}$ under the operator $\OP$ (cf. \autoref{lem1}) and \autoref{C0 estimate-step 1}, the limit
\eq{
\lim_{i\to \infty}\mathcal{A}_p^{\phi}((\OP)^i \Sigma)
}
exists and is finite and strictly positive. Furthermore, from \eqref{key} and the inequality
\eq{
\mathcal{A}_p^{\phi}((\OP)^i \Sigma)\leq \left(\frac{V(\widehat{(\OP)^{i+1} \Sigma})}{V(\widehat{(\OP)^i \Sigma})}\right)^{n-1} \mathcal{A}_p^{\phi}((\OP)^{i+1} \Sigma),
}
we deduce that
\eq{\label{volume estimate1}
\lim_{i\to\infty} \frac{V(\widehat{(\OP)^{i+1} \Sigma})}{V(\widehat{(\OP)^i \Sigma})} = 1.
}

By \autoref{uniform C1 estimate}, passing to a subsequence $\{i_j\}$, we obtain even, smooth, strictly convex capillary hypersurfaces $\cM_1$ and $\cM_2$ such that
\eq{\label{volume estimate2}
\lim_{j\to\infty} (\OP)^{i_j} \Sigma = \cM_1, \quad
\lim_{j\to\infty} (\OP)^{i_j+1} \Sigma = \cM_2.
}
Combining \eqref{volume estimate1} and \eqref{volume estimate2}, we find
\eq{
V(\widehat{\cM}_1) = \lim_{j\to\infty} V(\widehat{(\OP)^{i_j} \Sigma}) = \lim_{j\to\infty} V(\widehat{(\OP)^{i_j+1} \Sigma}) = V(\widehat{\cM}_2).
}
Moreover, since
\eq{
V(\widehat{(\OP)^{i_j} \Sigma}, \widehat{(\OP)^{i_j+1} \Sigma}[n-1]) = V(\widehat{(\OP)^{i_j} \Sigma}), \quad
V(\widehat{(\OP)^{i_j} \Sigma}) \geq V(\widehat{(\OP)^{i_j+1} \Sigma}),
}
it follows that
\eq{
V(\widehat{\cM}_1, \widehat{\cM}_2[n-1]) = V(\widehat{\cM}_1), \quad V(\widehat{\cM}_1) \geq V(\widehat{\cM}_2).
}
By the equality case in the Minkowski inequality (cf. \cite[Thm. 1.1]{MWWX24}) and that $\cM_1,\cM_2$ are both even capillary hypersurfaces, we conclude that $\cM_1 = \cM_2=:\cM$. Hence, the curvature function of $\cM$ satisfies
\eq{
f_{\cM} = \frac{V(\widehat{\cM})}{\frac{1}{n} \int_{\cC_\theta} \phi s_{\cM}^p \, d\sigma} \, \phi s_{\cM}^{p-1}.
}

\section*{Acknowledgment}
Hu was supported by the National Key Research and Development Program of China 2021YFA1001800 and the Fundamental Research Funds for the Central Universities. Ivaki was supported by the Austrian Science Fund (FWF) under Project P36545.

\providecommand{\bysame}{\leavevmode\hbox to3em{\hrulefill}\thinspace}

	\vspace{10mm}
	\textsc{School of Mathematical Sciences, Beihang University, Beijing 100191, China,}
	\email{\href{mailto:huyingxiang@buaa.edu.cn}{huyingxiang@buaa.edu.cn}}

	\vspace{3mm}
\textsc{Institut f\"{u}r Diskrete Mathematik und Geometrie,\\ Technische Universit\"{a}t Wien, Wiedner Hauptstra{\ss}e 8-10,\\ 1040 Wien, Austria,} \email{\href{mailto:mohammad.ivaki@tuwien.ac.at}{mohammad.ivaki@tuwien.ac.at}}

\end{document}